\documentclass[preprint,12pt]{elsarticle}

\usepackage{amssymb}
\usepackage{amsmath}

\newcommand{\m}{\mathrm{m}}
\newcommand{\M}{\mathrm{M}}
\DeclareMathOperator{\ex}{ex}

\usepackage{amsthm}

\newtheorem{definition}{Definition}[section]
\newtheorem{lemma}[definition]{Lemma}
\newtheorem{corollary}[definition]{Corollary}
\newtheorem{example}[definition]{Example}
\newtheorem{theorem}[definition]{Theorem}
\newtheorem{proposition}[definition]{Proposition}
\newtheorem{conjecture}[definition]{Conjecture}

\usepackage{ytableau}
\usepackage{colortbl}
\usepackage{comment}

\begin{document}

\begin{frontmatter}



\title{Pattern Forcing (0,1)-Matrices} 


\author{Lei Cao\textsuperscript{a}, Shen-Fu Tsai\textsuperscript{b}} 

\affiliation{organization={Department of Mathematics, Halmos College of Arts and Sciences, Nova Southeastern University},
            state={FL},
            postcode={33314},
            country={USA}}

\affiliation{organization={Department of Mathematics, College of Science, National Central University},
            city={Taoyuan},
            postcode={320317},
            country={Taiwan}}

\begin{abstract}
We introduce two related notions of \emph{pattern enforcement} in $(0,1)$-matrices—\emph{$Q$-forcing} and \emph{strongly $Q$-forcing}—which formalize distinct ways a fixed pattern $Q$ must appear within a larger matrix. A matrix is $Q$-forcing if every submatrix can realize $Q$ after turning any number of $1$-entries into $0$-entries, and strongly $Q$-forcing if every $1$-entry belongs to a copy of $Q$. 

For $Q$-forcing matrices, we establish the existence and uniqueness of extremal constructions minimizing the number of $1$-entries, characterize them using Young diagrams and corner functions, and derive explicit formulas and monotonicity results. For strongly $Q$-forcing matrices, we show that the minimum possible number of $0$-entries of an $m\times n$ strongly $Q$-forcing matrix is always $O(m+n)$, determine the maximum possible number of $1$-entries of an $n\times n$ strongly $P$-forcing matrix for every $2\times2$ and $3\times3$ permutation matrix, and identify symmetry classes with identical extremal behavior. 

We further propose a conjectural formula for the maximum possible number of $1$-entries of an $n\times n$ strongly $I_k$-forcing matrix, supported by results for $k=2,3$. These findings reveal contrasting extremal structures between forcing and strongly forcing, extending the combinatorial understanding of pattern embedding in $(0,1)$-matrices.
\end{abstract}



\begin{keyword}
$(0,1)$-matrices; $Q$-forcing; Strong $Q$-forcing.



  \MSC[2020] 05D99

\end{keyword}

\end{frontmatter}



\section{Introduction}
A $(0,1)$-matrix is a matrix in which each entry is either $0$ or $1$. All matrices discussed in this paper are $(0,1)$-matrices. There have been various extremal theories concerning pattern avoidance and pattern forcing in $(0,1)$-matrices. 

For pattern avoidance, a major line of research focuses on $\ex(n,P)$, the maximum possible number of $1$-entries that an $n\times n$ matrix can have without containing a given $(0,1)$-matrix $P$. A matrix $A$ is said to contain a matrix $B$ if $B$ can be obtained from $A$ by deleting any number of rows and columns followed by turning any number of $1$-entries into $0$-entries. 

This line of study originated from geometric and computational problems. Mitchell proposed an algorithm to find the shortest $L_1$ path between two points in a rectilinear grid with obstacles and bounded its time complexity by $\ex(n,P)$ for a certain pattern $P$~\cite{mitchell1992}, where $\ex(n,P)$ was first obtained by Bienstock and Gy\H{o}ri~\cite{BG1991}. Erd\H{o}s and Moser posed the question of finding the maximum number of unit distances among vertices of a convex $n$-gon~\cite{EM1959}, and F\"uredi gave the first upper bound $O(n\log_2 n)$, which improves upon $n^{1+\epsilon}$, using $\ex(n,P)$ with a suitable pattern $P$~\cite{furedi1990maximum}. 

Perhaps the most renowned application is the resolution of the Stanley--Wilf conjecture in enumerative combinatorics, which states that the number of permutations of $[n]$ avoiding a fixed permutation of $[k]$ is $O(c^n)$ for some $c>0$~\cite{SW}. Klazar showed that the conjecture holds if $\ex(n,P)$ is linear in $n$ for every permutation matrix $P$~\cite{klazar2000}, that is, if the F\"uredi--Hajnal conjecture~\cite{FH1992} is true. Marcus and Tardos later proved the F\"uredi--Hajnal conjecture, thereby settling the Stanley--Wilf conjecture~\cite{MT2003}. 

Other related works include the determination of asymptotic behaviors of $\ex(n,P)$ for small patterns $P$~\cite{FH1992,Tardos2005,Keszegh2009,Pettie2011a}, its connection with generalized Davenport--Schinzel sequences~\cite{Pettie2011a,Keszegh2009}, and generalizations to multidimensional matrices~\cite{KM2007,GT2017,Geneson2019,GT2025,Geneson2021b}. 

While pattern avoidance has been extensively studied, pattern forcing has received far less attention until recently, when Brualdi and Cao initiated the study of square $(0,1)$-matrices in which every permutation matrix contained covers certain subpatterns~\cite{BC2023}. However, the broader notion of $(0,1)$-matrices containing a prescribed submatrix has not yet been explored. Indeed, there has been no formal definition of pattern-forcing $(0,1)$-matrices. 

In this work, we introduce and analyze two related notions---\emph{forcing} and \emph{strongly forcing}---for embedding a fixed $(0,1)$-pattern $Q$ within a larger ambient matrix. Our goal is to investigate the extremal properties of matrices that guarantee the appearance of such a pattern.

\begin{definition}
Let $A$ be an $m \times n$ $(0,1)$-matrix and $Q$ an $s \times t$ $(0,1)$-matrix with $m\ge s,n\ge t$.  
We say that $A$ is \emph{$Q$-forcing} if every $s \times t$ submatrix of $A$ can be transformed into $Q$ by turning any number of its $1$-entries to $0$-entries.
\end{definition}

\begin{definition}
Let $A$ be an $m \times n$ $(0,1)$-matrix and $Q$ an $s \times t$ $(0,1)$-matrix with $m\ge s,n\ge t$.  
We say that $A$ is \emph{strongly $Q$-forcing} if for every $1$-entry $o$ of $A$, $A$ has an $s \times t$ submatrix \emph{exactly equal} to $Q$ that contains $o$.
\end{definition}

Intuitively, forcing requires that every submatrix \emph{could} realize $Q$, while strongly forcing requires that every $1$-entry in $A$ is \emph{witnessed} inside a copy of $Q$. These two perspectives highlight different ways patterns propagate within larger matrices.

\medskip
For a fixed $s \times t$ $(0,1)$-matrix $Q$, 
we focus on two extremal functions:
\begin{enumerate}
\item the \emph{minimum} number $\mathrm{m}(m,n,Q)$ of $1$-entries in an $m \times n$ $(0,1)$-matrix that is $Q$-forcing;
\item the \emph{maximum} number $\mathrm{M}(m,n,Q)$ of $1$-entries in an $m \times n$ $(0,1)$-matrix that is strongly $Q$-forcing.
\end{enumerate}

When $m=n$,
we use the abbreviated notations $\m(n,Q)$ and $\M(n,Q)$ instead.
If the pattern $Q$ is all-one,  all-zero, or has the same dimensions as the ambient matrix, then the extremal functions can be evaluated directly:
\begin{itemize}
\item If $Q$ is an all-one matrix, the unique $Q$-forcing and strongly $Q$-forcing $m\times n$ matrices are both all-one, giving
\[
    \mathrm{m}(m,n,Q) = \mathrm{M}(m,n,Q) = mn.
\]
\item If $Q$ is an all-zero matrix, the unique $Q$-forcing and strongly $Q$-forcing $m\times n$ matrices are both all-zero, so
\[
    \mathrm{m}(m,n,Q) = \mathrm{M}(m,n,Q) = 0.
\]
\item If $Q$ is $m\times n$, then
\[
    \mathrm{m}(m,n,Q) = \mathrm{M}(m,n,Q) = |Q|,
\]
where $|Q|$ denotes the number of $1$-entries in $Q$.
\end{itemize}

Throughout the paper, we do not consider any all-zero pattern unless stated otherwise. The remainder of the paper is organized as follows. In Section~2 we study $Q$-forcing matrices in full generality, emphasizing the rectangular case $\m(m,n,Q)$. We prove the existence and uniqueness of extremal constructions and express $\m(m,n,Q)$ in terms of combinatorial objects such as Young diagrams, leading to exact formulas and monotonicity results. Section~3 turns to strongly $Q$-forcing matrices, where the extremal problem is quite different: here we show that $\M(m,n,Q)$ is always within $\O(m+n)$ of the trivial
upper bound $mn$, and we determine exact values for small patterns, notably all $2\times 2$ and $3\times 3$ permutation matrices. For larger $k$, we provide constructions, identify dihedral symmetry classes of permutation matrices with equal extremal behavior, and propose a conjectural formula for $\M(n,I_k)$. Finally, Section~4 summarizes our findings, highlights the contrasting extremal behaviors of forcing versus strong forcing, and discusses directions for future investigation.

\section{$Q$-Forcing} \label{sec:Qforcing}

In this section, we present general results concerning $Q$-forcing $(0,1)$-matrices.
We begin with the unique $Q$-forcing matrix with the minimum number of $1$-entries.

\begin{lemma}\label{lemma:unique}
Let $Q$ be an $s \times t$ $(0,1)$-matrix, and let $m \geq s$, $n \geq t$.
There exists a unique $m \times n$ $Q$-forcing matrix $A$ with the minimum number of $1$-entries.
Moreover, $A$ can be constructed by setting to $1$ every entry that is required to be a $1$-entry in each $s \times t$ submatrix formed by $s$ consecutive rows and $t$ consecutive columns of $A$.
\end{lemma}

\begin{proof}
We construct a $(0,1)$-matrix $A'$ as follows.
Start with an $m \times n$ all-zero matrix. For every $s \times t$ submatrix of $A'$, set to $1$ each entry corresponding to a $1$-entry in $Q$ if it has not been set to $1$.
By definition of $Q$-forcing, any $m \times n$ $Q$-forcing matrix $B$ must satisfy $B \geq A'$ entrywise.
Hence $A'$ is the unique $m \times n$ $Q$-forcing matrix with the minimum number of $1$’s.

It remains to show that the matrix $A$ produced by the algorithm described in the statement is equal to $A'$. Clearly we have $A'\ge A$ entrywise, so it suffices to show that $A'\le A$ entrywise.

Suppose that $e = A'_{r_y,c_x}$ is set to $1$ while processing a submatrix of $A'$ determined by rows $r_1 < r_2 < \dots < r_s$ and columns $c_1 < c_2 < \dots < c_t$, with $Q_{y,x} = 1$ where $1 \le y \le s$, $1 \le x \le t$.
Then $e$ also lies in the submatrix consisting of rows
\[
    r_y - y + 1, \; r_y - y + 2, \; \dots, \; r_y - y + s
\]
and columns
\[
    c_x - x + 1, \; c_x - x + 2, \; \dots, \; c_x - x + t,
\]
where again the $(y,x)$-entry corresponds to $Q_{y,x} = 1$.
Thus $A_{r_y,c_x}$, the counterpart of $A'_{r_y,c_x}$, is also set to $1$ by the algorithm in the statement.
\end{proof}

Lemma~\ref{lemma:unique} implies that $\mathrm{m}(m,n,Q)$ is monotonic with respect to entrywise inequalities between matrices.  

\begin{corollary}
Suppose that $m\ge s,n\ge t$.
Let \(P\) and \(Q\) be $s\times t$ matrices with \(P \leq Q\) entrywise. If $A_P$ and $A_Q$ are the $m\times n$ $(0,1)$-matrices that is $P$-forcing and $Q$-forcing with minimum number of $1$-entries, respectively. Then $A_P\le A_Q$ entrywise and thus
\[
    \mathrm{m}(m,n,P) \leq \mathrm{m}(m,n,Q).
\]
\end{corollary}
However, in general monotonicity may not hold with respect to containment between patterns of different dimensions.
Here is an example.

\begin{example} Let
\[
    Q_1 = 
    \begin{bmatrix}
    1
    \end{bmatrix}, \quad
    Q_2 = 
    \begin{bmatrix}
    1 & 0 \\
    0 & 0
    \end{bmatrix}, \quad
    Q_3 = 
    \begin{bmatrix}
    1 & 1 & 1 & 1 \\ 
    1 & 1 & 0 & 1 \\
    1 & 0 & 0 & 1 \\
    1 & 1 & 1 & 1
    \end{bmatrix}.
\]

Observe that \(Q_2\) contains \(Q_1\) and that \(Q_3\) contains \(Q_2\). Moreover,
\[
\mathrm{m}(m,n,Q_1) = mn > (m-1)(n-1) = \mathrm{m}(m,n,Q_2)
\]
and
\[
\mathrm{m}(m,n,Q_3) > \mathrm{m}(m,n,Q_2),
\]
where \(m \ge 4\) and \(n \ge 4\).

\end{example}

\medskip

Let $Q$ be an $r \times s$ $(0,1)$-matrix.
If $m \geq 2r$ and $n \geq 2s$, we will describe the structure of the unique $m \times n$ $Q$-forcing matrix with $\mathrm{m}(m,n,Q)$ $1$-entries. We say that a position $(i_1,j_1)$ \emph{dominates} a position $(i_2,j_2)$ if $i_1 \ge i_2$ and $j_1 \ge j_2$.
We say that $(i_1,j_1)$ \emph{alt-dominates} $(i_2,j_2)$ if $i_1 \le i_2$ and $j_1 \ge j_2$.
Every position both dominates and alt-dominates itself, and each relation defines a nonstrict partial order.
Let $\mathcal{P}$ be a set of positions and let $p \in \mathcal{P}$.
We call $p$ an \emph{(alt-)minimal} position of $\mathcal{P}$ if it does not (alt-)dominate any other position in $\mathcal{P}$.
Likewise, $p$ is an \emph{(alt-)maximal} position of $\mathcal{P}$ if it is not (alt-)dominated by any other position in $\mathcal{P}$.

\ytableausetup{centertableaux}

To aid understanding, we illustrate our results using Young diagrams, beginning with a brief review of the connection between $(0,1)$-matrices and Young diagrams.
A \textbf{Young diagram} (or Ferrers diagram) is a finite collection of boxes arranged in left-justified rows with row lengths in non-increasing order from top to bottom.
It is often used as a visual representation of a partition of a positive integer.
The \emph{cardinality} of a Young diagram is the total number of boxes it contains.

\begin{example}
Consider the partition
\[
    11 = 4 + 3 + 3 + 1.
\]
Its corresponding Young diagram is
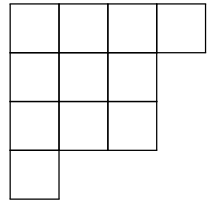
\begin{ytableau}
\none &  &  &  & \\
\none &  &  &  \\
\none &  &  &  \\
\none &
\end{ytableau}
\end{example}

Let $Q$ be an $s \times t$ $(0,1)$-matrix. Define $NW(Q)$ as the largest Young diagram that can be placed in the upper left corner of $Q$ without covering any $1$-entries, and let $|NW(Q)|$ denote its cardinality.
Similarly, define $SW(Q)$ as the largest Young diagram that can be placed in the lower left corner of $Q$ after being flipped vertically.
Define $NE(Q)$ as the largest Young diagram that can be placed in the upper right corner of $Q$ after being flipped horizontally.
Define $SE(Q)$ as the largest Young diagram that can be placed in the lower right corner of $Q$ after being flipped vertically and horizontally. We call $NW(Q)$, $NE(Q)$, $SE(Q)$, and $SW(Q)$ the \emph{corner functions} of $(0,1)$-matrix $Q$.

\medskip
The corner functions can be equivalently defined as follows.
\begin{itemize}
\item $NW(Q)$ is the set of positions of $0$-entries that do not dominate any $1$-entries in $Q$;
\item $SW(Q)$ is the set of positions of $0$-entries that do not alt-dominate any $1$-entries in $Q$;
\item $NE(Q)$ is the set of positions of $0$-entries that are not alt-dominated by any $1$-entries in $Q$;
\item $SE(Q)$ is the set of positions of $0$-entries that are not dominated by any $1$-entries in $Q$.
\end{itemize}

Using the corner functions of $Q$  we characterize an $m\times n$ $(0,1)$-matrix that will be proved equal to the output of the algorithm in Lemma~\ref{lemma:unique}. 

\begin{definition}
For a $s\times t$ $(0,1)$-matrix $Q$ and $m \ge 2s,n \ge 2t$, let $A_{m,n,Q}$ be an $m\times n$ $(0,1)$-matrix satisfying the following. Let $A_{NW},A_{NE},A_{SE},A_{SW}$ be the $s\times t$ upper left, upper right, lower right, and lower left submatrices of $A_{m,n,Q}$, respectively, such that
\begin{enumerate}
\item $NW(A_{NW})=NW(Q)$,
\item $NE(A_{NE})=NE(Q)$,
\item $SE(A_{SE})=SE(Q)$,
\item $SW(A_{SW})=SW(Q)$,
\item $A_{m,n,Q}$ has the same number of consecutive all-zero top rows as $Q$,
\item $A_{m,n,Q}$ has the same number of  consecutive all-zero bottom rows as $Q$,
\item $A_{m,n,Q}$ has the same number of  consecutive all-zero leftmost columns as $Q$,
\item $A_{m,n,Q}$ has the same number of consecutive all-zero rightmost columns as $Q$, and
\item $A_{m,n,Q}$ has no other $0$-entries.
\end{enumerate}
\end{definition}

\begin{example}
Let $Q$ be the $7 \times 6$ $(0,1)$-matrix.
The regions $NW(Q)$, $SW(Q)$, $NE(Q)$, and $SE(Q)$ are highlighted in yellow, red, green, and blue, respectively.
We have
\[
    |NW(Q)| = 7, \quad |SW(Q)| = 4, \quad |NE(Q)| = 8, \quad |SE(Q)| = 2.
\]

\[
    \left[\begin{array}{c|c|c|c|c|c}
    \cellcolor{yellow!20}&\cellcolor{yellow!20}&1&\cellcolor{green!20}&\cellcolor{green!20}&\cellcolor{green!20}\\ \hline
    \cellcolor{yellow!20}&\cellcolor{yellow!20}&&1&\cellcolor{green!20}&\cellcolor{green!20}\\ \hline
    \cellcolor{yellow!20}&1&1&1&\cellcolor{green!20}&\cellcolor{green!20}\\ \hline
    \cellcolor{yellow!20}&&&1&1&\cellcolor{green!20}\\ \hline
    \cellcolor{yellow!20}&1&&&1&1\\ \hline
    1&&1&&&\cellcolor{blue!20}\\ \hline
    \cellcolor{red!20}&\cellcolor{red!20}&\cellcolor{red!20}&\cellcolor{red!20}&1&\cellcolor{blue!20}\\
    \end{array}\right]
\]
and
\[
    A_{14,12,Q} =
    \left[\begin{array}{c|c|c|c|c|c|c|c|c|c|c|c}
    0\cellcolor{yellow!20}&0\cellcolor{yellow!20}&1&1&1&1&1&1&1&0\cellcolor{green!20}&0\cellcolor{green!20}&0\cellcolor{green!20}\\ \hline
    0\cellcolor{yellow!20}&0\cellcolor{yellow!20}&1&1&1&1&1&1&1&1&0\cellcolor{green!20}&0\cellcolor{green!20}\\ \hline
    0\cellcolor{yellow!20}&1&1&1&1&1&1&1&1&1&0\cellcolor{green!20}&0\cellcolor{green!20}\\ \hline
    0\cellcolor{yellow!20}&1&1&1&1&1&1&1&1&1&1&0\cellcolor{green!20}\\ \hline
    0\cellcolor{yellow!20}&1&1&1&1&1&1&1&1&1&1&1\\ \hline
    1&1&1&1&1&1&1&1&1&1&1&1 \\ \hline
    1&1&1&1&1&1&1&1&1&1&1&1 \\ \hline
    1&1&1&1&1&1&1&1&1&1&1&1 \\ \hline
    1&1&1&1&1&1&1&1&1&1&1&1 \\ \hline
    1&1&1&1&1&1&1&1&1&1&1&1 \\ \hline
    1&1&1&1&1&1&1&1&1&1&1&1 \\ \hline
    1&1&1&1&1&1&1&1&1&1&1&1 \\ \hline
    1&1&1&1&1&1&1&1&1&1&1&0\cellcolor{blue!20}\\ \hline
    0\cellcolor{red!20}&0\cellcolor{red!20}&0\cellcolor{red!20}&0\cellcolor{red!20}&1&1&1&1&1&1&1&0\cellcolor{blue!20}\\
    \end{array}\right]
\]
\end{example}

The four corner functions of a $(0,1)$-matrix $Q$ may not be disjoint.
For example, if the top row of $Q$ is all-zero then both $NW(Q)$ and $NE(Q)$ contain all positions of the top row.

\begin{theorem}
Let $Q$ be an $s \times t$ $(0,1)$-matrix, $m \ge 2s$, $n \ge 2t$, and $Q$ is not all-zero. Then the algorithm in Lemma~\ref{lemma:unique} produces $A_{m,n,Q}$.
\end{theorem}

\begin{proof}
First, we show that no position of any $0$-entry of $A_{m,n,Q}$ is set to $1$ by the algorithm. If the position of a $0$-entry at $(i,j)$ in the upper left $s\times t$ submatrix of $A_{m,n,Q}$ is set to $1$ by the algorithm, then it corresponds to a $1$-entry of $Q$ when we shift $Q$ down any number of rows and right any number of columns from the upper left corner of the ambient matrix. Therefore  $(i,j)$ dominates a $1$-entry of $Q$, contradicting with the definition of $NW(Q)$. Similarly no position of any $0$-entry in the upper right, lower left, and lower right $s\times t$ submatrices of $A_{m,n,Q}$ is set to $1$ by the algorithm. Also, it is clear that no position of any $0$-entry in the top consecutive all-zero rows of $A_{m,n,Q}$ is set to $1$ by the algorithm, and nor is the position of any $0$-entry in the bottom consecutive all-zero rows, leftmost consecutive all-zero columns, and rightmost consecutive all-zero columns of $A_{m,n,Q}$.

Next, we show that the position of every $1$-entry of $A_{m,n,Q}$ is set to $1$ by the algorithm. For a $1$-entry at $(i,j)$ in the upper left $s\times t$ submatrix of $A_{m,n,Q}$, since $(i,j)\notin NW(Q)$ it dominates a $1$-entry of $Q$. So position $(i,j)$ is set to $1$ by the algorithm when shifting $Q$ down any number of rows and right any number of columns from the upper left corner of the ambient matrix. Similarly the position of every $1$-entry in the upper right, lower left, and lower right $s\times t$ submatrix of $A_{m,n,Q}$ is set to $1$ by the algorithm.
Also, the position of every remaining entry of the ambient matrix is clearly set to $1$ by the algorithm.
\end{proof}

Together with the corner functions of $Q$,
we can then express $\mathrm{m}(m,n,Q)$ in terms of the number of consecutive all-zero top rows, consecutive all-zero bottom rows, consecutive all-zero leftmost columns, and consecutive all-zero rightmost columns. Alternatively, it is conceptually intuitive to strip off all-zero columns and rows at the boundary of $Q$ and treat the remaining as its core. 
Let $Q_{core}$ denote the $s' \times t'$ $(0,1)$-matrix obtained by deleting all-zero rows from the top and bottom and all-zero columns from the left and right until each of the boundary rows and columns contains at least one $1$-entry.
Clearly $s' \le s$ and $t' \le t$.
Note that $Q_{core}$ may contain all-zero rows or columns which were not at the boundary of $Q$.

\begin{example}
\[
    Q=\left[\begin{array}{c|c|c|c|c}
    0 & 0& 0 & 0 & 0\\ \hline
    0 & 0& 0 & 0 & 0 \\ \hline
    0 & 0& 0 & 0 & 0\\ \hline
    1 & 0& 0& 1 & 0\\ \hline
    0 & 0 & 0 & 0& 0\\ \hline
    1& 1& 1 & 0 & 0 \\ \hline
    1&1 & 1& 0 &0
    \end{array}\right],
    \quad
    Q_{core}=\left[\begin{array}{c|c|c|c}
    1 & 0& 0& 1\\ \hline
    0 & 0 & 0 & 0\\ \hline
    1& 1& 1 & 0  \\ \hline
    1&1 & 1& 0
    \end{array}\right]
\]
\end{example}

By counting the $0$-entries of $A_{m,n,Q}$, we have the following result.
\begin{corollary}
Let $Q$ be an $s \times t$ $(0,1)$-matrix, its core $Q_{core}$ be an $s'\times t'$ $(0,1)$-matrix, and let $m \ge 2s$ and $n \ge 2t$. Then
\begin{equation*}
\begin{split}
\m(m,n,Q)=mn&-(m-2s)(t-t')-(n-2t)(s-s')\\
&-|NW(Q)|-|SW(Q)|-|NE(Q)|-|SE(Q)|
\end{split}
\end{equation*}
\end{corollary}

When $s=s',t=t'$, the following statement follows.
\begin{corollary}\label{cor:nonempty}
Let $Q$ be an $s \times t$ $(0,1)$-matrix such that row $1$, row $s$, column $1$, and column $t$ each contain at least one $1$-entry, and let $m \ge 2s$ and $n \ge 2t$. Then
$$
    \mathrm{m}(m,n,Q)=mn - \left(|NW(Q)| + |SW(Q)| + |NE(Q)| + |SE(Q)|\right).
$$
\end{corollary}

We can also express $\m(m,n,Q)$ in terms of the corner functions of $Q_{core}$.
\begin{theorem}
Let \(Q\) be an \(s \times t\) \((0,1)\)-matrix, and let the $s'\times t'$ matrix \(Q_{\mathrm{core}}\) denote its core. 
If  $m - (s - s') \ge 2s'$ and $n - (t - t') \ge 2t'$, then
\begin{equation*}
\begin{split}
    \mathrm{m}(m,n,Q) 
    =& \left(m - (s-s')\right)\left(n - (t-t')\right) \\
    &- \left(|NW(Q_{\mathrm{core}})| + |SW(Q_{\mathrm{core}})| + |NE(Q_{\mathrm{core}})| + |SE(Q_{\mathrm{core}})|\right).
\end{split}    
\end{equation*}
\end{theorem}

\begin{proof}
By Lemma~\ref{lemma:unique}, the unique $m\times n$ $Q$-forcing matrix with $\m(m,n,Q)$ $1$-entries is obtained setting all necessary entries to $1$ while moving $Q$ around in the $m\times n$ ambient matrix. It is equivalent to setting all necessary entries to $1$ while moving $Q_{core}$ around in the corresponding $\left(m-(s-s')\right)\times\left(n-(t-t')\right)$ submatrix of the $m\times n$ ambient matrix. Since $Q_{core}$ has $1$-entries in its first and last rows and its first and last columns, the result thus follows from Corollary~\ref{cor:nonempty}.
\end{proof}

As permutation matrices form an important class of $(0,1)$-matrices, we now turn our attention to permutation patterns and apply Corollary~\ref{cor:nonempty} as a powerful tool.

\begin{lemma}
Let $n \ge 2k$ and let $P$ be a $k \times k$ permutation matrix. Then
$$
    \mathrm{m}(n,P) \ge n^2 - k(k-1).
$$
The equality holds
if and only if $P$ is $I_k$ or its transpose $I_k^t$.
\end{lemma}

\begin{proof}
For $k=2$, the claim is immediate. Suppose that $k \ge 3$.

Every permutation matrix $P$ has exactly one $1$ in each row and each column. By Corollary~\ref{cor:nonempty}, we have
$$
    \mathrm{m}(n,P) = n^2 - \left(|NW(P)| + |SW(P)| + |NE(P)| + |SE(P)|\right).
$$

Since every $k \times k$ permutation matrix has exactly $k(k-1)$ $0$-entries and the corner functions of $P$
are pairwise disjoint, we have
$$
    |NW(P)| + |SW(P)| + |NE(P)| + |SE(P)|\le k(k-1)
$$
and the inequality follows. The equality clearly holds when $P$ is $I_k$ or $I_k^t$.
If $P$ is a $k \times k$ permutation matrix not equal to $I_k$ or $I_k^t$, then $P$ has a submatrix sliced from three consecutive columns and three consecutive rows that is equal to one of the following $3 \times 3$ permutation matrices:
\[
    P_1=\begin{bmatrix}
    1&0&0\\
    0&\cellcolor{yellow!20}0&1 \\
    0&1&0
    \end{bmatrix},\quad
    P_2=\begin{bmatrix}
    0&1&0\\
    0&\cellcolor{yellow!20}0&1 \\
    1&0&0
    \end{bmatrix},\quad
    P_3=\begin{bmatrix}
    0&1&0\\
    1&\cellcolor{yellow!20}0&0 \\
    0&0&1
    \end{bmatrix},\quad
    P_4=\begin{bmatrix}
    0&0&1\\
    1&\cellcolor{yellow!20}0&0 \\
    0&1&0
    \end{bmatrix}.
\]
In each case, the central $0$-entry, highlighted in yellow, does not belong to any of the corner functions of $P$. Consequently
$$
    |NW(P)| + |SW(P)| + |NE(P)| + |SE(P)|< k(k-1)
$$
and hence $\m(n,P)>n^2-k(k-1)$.
\end{proof}

\noindent
We also find the $k\times k$ permutation matrices that maximizes $\m(n,P)$.
First, we define the \emph{perimeter} of a matrix as the union of its first and last rows together with its first and last columns.

\begin{lemma}
Let $n \ge 2k$, and let $P=[p_{ij}]$ be a $k \times k$ permutation matrix. Then
\[
    \max_P \, \mathrm{m}(n,P) =
    \begin{cases}
    n^2 & k=1, \\[6pt]
    n^2-2 & k=2, \\[6pt]
    n^2-5 & k=3, \\[6pt]
    n^2-4k+8 & k \ge 4.
    \end{cases}
\]
For $k \ge 4$, this maximum is attained if and only if either
\[
    p_{1,2}, \; p_{2,k}, \; p_{k,k-1}, \; p_{k-1,1} \quad \text{are all $1$-entries},
\]
or
\[
    p_{2,1}, \; p_{k,2}, \; p_{k-1,k}, \; p_{1,k-1} \quad \text{are all $1$-entries}.
\]
\end{lemma}

\begin{proof}
We consider $k\ge4$ since the cases $k\le3$ are straightforward to verify.

For the first part, observe that the position of every $0$-entry on the perimeter of $P$ belongs to one of the corner functions of $P$.
Since the perimeter of $P$ has at least $4k-8$ $0$-entries, we have
$$
    \m(n,P)\le n^2-(4k-8).
$$
Equality cannot hold if any corner of $P$ is $1$-entry. For example, if both the upper right and lower left corners are $1$-entries, then the perimeter of $P$ contains $4k-6$ $0$-entries, implying
\[
    \mathrm{m}(n,P) \;\le\; n^2 - (4k-6) \;<\; n^2 - (4k-8).
\]

For the second part, suppose that $P$ is a $k \times k$ and $\mathrm{m}(n,P) = n^2 - 4k + 8$. None of the four corners of $P$ can be $1$, or else the perimeter would contain at least $4k-7$ $0$-entries. The upper left $2\times2$ submatrix of $P$ cannot be all-zero, or else $(2,2)\in NW(P)$ and the maximality of $\m(n,P)$ is violated. Similarly none of the upper right, lower left, and lower right $2\times2$ submatrices of $P$ is all-zero.
We claim that at least one of
$p_{1,2}$ and $p_{2,1}$ is $1$-entry. If they are both $0$-entries, then $p_{2,2}=1$ or else the upper left $2\times2$ submatrix of $P$ is all-zero. With $p_{2,2}=1$, we have $p_{k-1,1}=p_{1,k-1}=1$ and the lower right $2\times2$ submatrix of $P$ is all-zero, violating the maximality of $\m(n,P)$.
Without loss of generality, assume $p_{1,2}=1$ and we have 
\[
    p_{1,2}, \; p_{2,k}, \; p_{k,k-1}, \; p_{k-1,1}
\]
are all $1$-entries. By symmetry, the only other possibility for $\m(n,P)=n^2-4k+8$ is that
\[
    p_{2,1}, \; p_{k,2}, \; p_{k-1,k}, \; p_{1,k-1}
\]
are all $1$-entries.
\end{proof}

\section{Strongly $Q$-Forcing}\label{sec:StrongQforcing}

We begin with a general result for an arbitrary $(0,1)$-pattern $Q$: the number of $0$-entries of an $m\times n$ strongly $Q$-forcing matrix with maximum possible number of $1$-entries is at most linear in $m+n$.

\begin{lemma}
Let $Q$ be an $s\times t$ $(0,1)$-matrix. Then 
\[
    mn - \mathrm{M}(m,n,Q) = O(m+n).
\]
\end{lemma}

\begin{proof}
We construct an $m\times n$ strongly $Q$-forcing matrix $A$ with $O(m+n)$ $0$-entries.  
Start with $A=Q$. Suppose that the first $s_0$ rows of $Q$ are all-zero, and let the leftmost $1$-entry in row $s_0+1$ of $Q$ be in column $t_1$.  
Replace column $t_1$ of $A$ by $n-t+1$ copies of itself, then replace row $s_0+1$ of $A$ by $m-s+1$ copies of itself.  
The resulting $m\times n$ matrix $A$ is strongly $Q$-forcing and contains only $O(m+n)$ zero entries.
\end{proof}

In the following, for each $2\times 2$ permutation matrix $P$ we determine $\M(n,P)$ and the unique $n\times n$ strongly $P$-forcing matrix with this many $1$-entries. 
Denote by $J_n$ the $n\times n$ all-one matrix and by $H_n$ the $n\times n$ Hankel identity matrix that is the $n\times n$ $(0,1)$-matrix with $1$-entries on the anti-diagonal $(1,n),(2,n-1),\dots,(n,1)$ and $0$-entries elsewhere.

\begin{lemma} \label{lemma:MI2}
Let $n\ge2$
 and let $Q$ be a $2 \times 2$ permutation matrix. Then
\[
    \M(n, Q) = n^2 - n.
\]
Moreover, the unique strongly $I_2$-forcing $(0,1)$-matrix containing $n^2 - n$ $1$-entries is \( J_n - H_n \);
the unique strongly $H_2$-forcing $(0,1)$-matrix containing $n^2 - n$ $1$-entries is \( J_n - I_n \).
\end{lemma}

\begin{proof}
By symmetry it suffices to prove the result for $Q=I_2$.
Let $A$ be an $n\times n$ strongly $I_2$-forcing matrix. Observe that every $1$-entry of $A$ either has a $0$-entry above it in the same column and a $0$-entry to its left in the same row, or has a $0$-entry below it in the same column and a $0$-entry to its right in the same row.

For the first part, every row and every column of $I_2$ has one $0$-entry, so every row and every column of $A$ has one $0$-entry. Hence $A$ has at least $n$ $0$-entries and at most $n^2-n$ $1$-entries. The matrix $J_n - H_n$ is strongly $I_2$-forcing and has $n^2-n$ $1$-entries. Hence,
\[
    \M(n,I_2) = n^2 - n.
\]

For the second part, suppose that each row and each column of $A$ has exactly one $0$-entry.
By definition, the entry $(1,n)$ of $A$ must be a $0$-entry. Hence, the rest of row $1$ and column $n$ of $A$ are all-one. If the entry $(2,n-1)$ is not a $0$-entry, then it does not have a $0$-entry above it in the same column and a $0$-entry to its left in the same row, and it does not have a $0$-entry below it in the same column and a $0$-entry to its right in the same row. Therefore the entry $(2,n-1)$ of $A$ is a $0$-entry and the rest of row $2$ and column $n-1$ are all-one. It could be seen that continuing this reasoning we can conclude that all $0$-entries of $A$ lie on its anti-diagonal, i.e., $A=J_n-H_n$.
\end{proof}

Below we give a recurrence for the lower bound on $\M(n,I_k)$. 
\begin{proposition} \label{pro1}
For $1\le k\le n$,
$$
    \M(n,I_k)\ge v
$$
where
$$
v=\max\left\{\M(n_1,I_{k_1})+\M(n_2,I_{k_2})\right\}\text{ subject to }
\begin{cases}
    1\le n_1,n_2\le n\\
    k_1+k_2=k\\
    n_1+n_2=n\\
    1\le k_1\le n_1\\
    1\le k_2\le n_2\\
\end{cases}.
$$
\end{proposition}
\begin{proof}
    Given an $n_1\times n_1$ strongly $I_{k_1}$-forcing matrix $A_1$ with $\M(n_1,I_{k_1})$ $1$-entries and an $n_2\times n_2$ strongly $I_{k_2}$-forcing matrix $A_2$ with $\M(n_1,I_{k_1})$ $1$-entries, we obtain an $n\times n$ strongly $I_k$-forcing matrix with $\M(n_1,I_{k_1})+\M(n_2,I_{k_2})$ $1$-entries
    \[
    A = \begin{bmatrix}
    A_1 & 0 \\
    0 & A_2
    \end{bmatrix}.
\]
Therefore the result follows.
\end{proof}

$\M(n,I_3)$ is not as straightforward as $\M(n,I_2)$. The following result provides a lower bound on the number of $0$-entries of a $n\times n$ strongly $I_3$-forcing matrix.

\begin{lemma}\label{lm123}
For $n \geq 3$, we have $\M(n, I_3) \leq n^2 - 3n + 3$.
\end{lemma}

\begin{proof}
Let $Q = [q_{ij}]$ denote the matrix $I_3$, and let $A = [a_{ij}]$ be an $n \times n$ matrix that is strongly $Q$-forcing. Since every row and column of $Q$ contains two $0$-entries, every row and column of $A$ must have at least two $0$-entries.

If every column of $A$ contains at least three $0$-entries, then $A$ has at most $n^2 - 3n$ $1$-entries, and the result follows. Now suppose that some column $c$ of $A$ contains exactly two $0$-entries, denoted by $z_1 = a_{r_1,c}$ and $z_2 = a_{r_2,c}$ with $r_1 < r_2$.

Let $E$ be the total number of \emph{extra $0$-entries}, i.e., $0$-entries beyond the first two occurring in any row or column, in every column and row. Note that if a zero is an extra zero in both its column and row, then we count it twice. Denote by $\min E$ a lower bound on $E$. Counting all $0$-entries in $A$ twice: row by row followed by column by column, we see that $A$ has at least $(2n+2n+\min E)/2$ $0$-entries. Thus we have
\[
    \M(n, I_3) \leq n^2 - \left(2n + \frac{1}{2} \min E\right).
\]

We consider the following two cases:
\begin{enumerate}

\item[$\mathbf{(1)}$] $r_1 + 1 < r_2$

Examine the $1$-entries between $z_1$ and $z_2$, starting from $a_{r_1 + 1, c}$. For each $r \in [r_1 + 1, r_2 - 1]$, let $d_r$ denote the column corresponding to the first column of $Q$, assuming $a_{r,c}$ corresponds to $q_{2,2}$. For each $r \in [r_1 + 2, r_2 - 1]$:

If $d_r \in \{d_s : s \in [r_1 + 1, r - 1]\}$, then column $d_r$ has an extra zero at $a_{r, d_r}$. Otherwise, row $r_2$ has an extra zero at $a_{r_2, d_r}$. These contribute a total of $r_2 - r_1 - 2$ to $E$.

Let $\mathcal{D} = \{d_r : r \in [r_1 + 1, r_2 - 1]\}$. Next, examine the $1$-entries below $z_2$ in column $c$, starting from $a_{r_2 + 1, c}$. For each $r \in [r_2 + 1, n]$, let $c_r$ denote the column corresponding to the second column of $Q$, assuming $a_{r,c}$ corresponds to $q_{3,3}$. Then:

If $c_r \in \mathcal{D} \cup \{c_s : s \in [r_2 + 1, r - 1]\}$, then column $c_r$ has an extra zero at $a_{r, c_r}$. Otherwise, row $r_1$ has an extra zero at $a_{r_1, c_r}$, since $z_1$ corresponds to $q_{1,2}$ and a $q_{1,3}$ zero is required to its right.

These contribute $n - r_2$ to $E$. A similar argument applies for the $0$-entries corresponding to $q_{2,1}$ or $q_{3,1}$, contributing another $n - r_2$.

By symmetry, the $0$-entries to the right of column $c$ contribute:
\[
    2(r_1 - 1) + (r_2 - r_1 - 2).
\]

Summing all contributions:
\[
    E \ge 2(n - r_2) + 2(r_1 - 1) + 2(r_2 - r_1 - 2) = 2(n - 3),
\]
and therefore:
\[
    \M(n, I_3) \le n^2 - 2n - (n - 3) = n^2 - 3n + 3.
\]

\item[$\mathbf{(2)}$] $r_1 + 1 = r_2$

If $z_1$ and $z_2$ are located at the top or bottom of column $c$, the same argument as above applies, yielding $E \ge 2(n - 3)$.

Suppose instead that $z_1$ and $z_2$ are not at the top or bottom. Then the one at $a_{r_1 - 1, c}$ (just above $z_1$) forces a zero to the right of $z_1$ corresponding to $q_{2,3}$.
Therefore $0$-entries corresponding to $q_{2,1}$ or $q_{3,1}$ for $1$-entries below row $r_2+1$ contribute $n - r_2 - 1$ to $E$. $0$-entries corresponding to $q_{1,2}$ or $q_{3,2}$ for $1$-entries below row $r_2$ contribute $n - r_2$.

Similarly, adding contributions from the $1$-entries above $z_1$ we get:
\[
    E\ge 2n - 2r_2 - 1 + 2r_1 - 3 = 2n - 6.
\]
\end{enumerate}

Thus, the desired bound follows.
\end{proof}

For the $3\times3$ permutation matrix $Q$ corresponding to the permutation $132$, the following result gives the same bound on $\M(n,Q)$ as $\M(n,I_3)$ using the same proof. However, there are subtleties in the proof that need to be taken care of, so for the completeness we include the whole proof below.
\begin{lemma} \label{lm132}
Let $Q$ be the $3\times3$ permutation matrix corresponding to permutation $132$, then
$$ \M(n, Q) \leq n^2 - 3n + 3.$$
\end{lemma}

\begin{proof}
Let $A = [a_{ij}]$ be an $n \times n$ matrix that is strongly $Q$-forcing. Since every row and column of $Q$ contains two $0$-entries, every row and column of $A$ must have at least two $0$-entries.

If every column of $A$ contains at least three $0$-entries, then $A$ has at most $n^2 - 3n$ $1$-entries, and the result follows. Now suppose that some column $c$ of $A$ contains exactly two $0$-entries, denoted by $z_1 = a_{r_1,c}$ and $z_2 = a_{r_2,c}$ with $r_1 < r_2$.

Let $E$ be the total number of \emph{extra $0$-entries}, i.e., $0$-entries beyond the first two occurring in any row or column, in every column and row. Note that if a zero is an extra zero in both its column and row, then we count it twice. Denote by $\min E$ a lower bound on $E$. Counting all $0$-entries in $A$ twice: row by row followed by column by column, we see that $A$ has at least $(2n+2n+\min E)/2$ $0$-entries. Thus we have
\[
    \M(n, Q) \leq n^2 - \left(2n + \frac{1}{2} \min E\right).
\]

We consider the following two cases:
\begin{enumerate}

\item[$\mathbf{(1)}$] $r_1 + 1 < r_2$

Examine the $1$-entries between $z_1$ and $z_2$, starting from $a_{r_1 + 1, c}$. For each $r \in [r_1 + 1, r_2 - 1]$, let $d_{3,r}$ and $f_{3,r}$ denote the columns corresponding to the first column and the second column of $Q$ respectively, assuming $a_{r,c}$ corresponds to $q_{2,3}$.

For each $r \in [r_1 + 2, r_2 - 1]$:

If $d_{3,r} \in \{d_{3,s} : s \in [r_1 + 1, r - 1]\}$, then column $d_{3,r}$ has an extra zero at $a_{r, d_r}$. Otherwise, row $r_2$ has an extra zero at $a_{r_2, d_r}$. These contribute a total of $r_2 - r_1 - 2$ to $E$.

Similarly, If $f_{3,r} \in \{f_{3,s} : s \in [r_1 + 1, r - 1]\}$, then column $f_{3,r}$ has an extra zero at $a_{r, f_{3,r}}$. Otherwise, row $r_1$ has an extra zero at $a_{r_1, f_{3,r}}$. These contribute a total of $r_2 - r_1 - 2$ to $E$. Note that there are forced $0$-entries to the left of $z_1$ and $z_2$ in rows $r_1$ and $r_2$ corresponding to $q_{1,2}$ and $q_{3,1}$, respectively.

Let $\mathcal{D}_3 = \{d_{3,r} : r \in [r_1 + 1, r_2 - 1]\}$. Next, examine the $1$-entries below $z_2$ in column $c$, starting from $a_{r_2 + 1, c}$. For each $r \in [r_2 + 1, n]$, let $d_{2,r}$ and $f_{2,r}$ denote the column corresponding to the first and third columns of $Q$, respectively, assuming $a_{r,c}$ corresponds to $q_{3,2}$.

Then:

If $d_{2,r} \in \mathcal{D}_3 \cup \{d_{2,s} : s \in [r_2 + 1, r - 1]\}$, then column $d_{2,r}$ has an extra zero at $a_{r, d_{r,2}}$. Otherwise, row $r_2$ has an extra zero at $a_{r_2, d_{2,r}}$. These contribute $n - r_2$ to $E$.

A similar argument applies for the $0$-entries corresponding to $q_{1,3}$ or $q_{3,3}$, contributing another $n - r_2$.

Similar arguments apply for the $1$-entries above $z_1$ in the column $c$, contributing $2(r_1-1)$ to $E$.

Summing all contributions:
\[
    E \ge 2(n - r_2) + 2(r_1 - 1) + 2(r_2 - r_1 - 2) = 2(n - 3),
\]
and therefore:
\[
    \M(n, Q) \le n^2 - 2n - (n - 3) = n^2 - 3n + 3.
\]

\item[$\mathbf{(2)}$] $r_1 + 1 = r_2$

If $z_1$ and $z_2$ are located at the top or bottom of column $c$, the same argument as above applies, yielding $E \ge 2(n - 3)$.

Suppose instead that $z_1$ and $z_2$ are not at the top or bottom. Then the one at $a_{r_1 - 1, c}$ (just above $z_1$) forces a zero to the right of $z_1$ corresponding to $q_{2,2}$ and a zero to the right of $z_2$ corresponding to $q_{3,3}$.
Therefore $0$-entries corresponding to $q_{2,1}$ or $q_{3,1}$ for $1$-entries below row $r_2$ contribute $n - r_2$ to $E$. $0$-entries corresponding to $q_{1,3}$ or $q_{3,3}$ for $1$-entries below row $r_2+1$ contribute $n - r_2 - 1$.

Similarly, adding contributions from the $1$-entries above $z_1$ we get:
\[
    E\ge 2n - 2r_2 - 1 + 2r_1 - 3 = 2n - 6.
\]
\end{enumerate}

Thus, the desired bound follows.
\end{proof}

It could be seen that certain operations on any matrix $Q$ do not affect $\M(n,Q)$.
We state the following result without proof.
\begin{lemma}\label{lemma:symmetry}
Let $Q$ be an $s\times t$ $(0,1)$-matrix, and let $Q'$ be obtained from $Q$ by performing any sequence of 
horizontal reflection (row reversal), vertical reflection (column reversal), and transposition.
Then for every $n$,
\[
    \mathrm{M}(n,Q') = \mathrm{M}(n,Q).
\]
\end{lemma}

\medskip

Let $Q$ be an $m\times m$ permutation matrix. If $Q'$ can be obtained by applying any combination of the operations in Lemma~\ref{lemma:symmetry}, then $Q$ and $Q'$ lie in the same \emph{dihedral symmetry class}. Denote this class by $\mathfrak{D}$. Then, for all $n$, Lemma~\ref{lemma:symmetry} implies
\[
    \mathrm{M}(n,Q_1) = \mathrm{M}(n,Q_2) \quad \text{for every } Q_1,Q_2 \in \mathfrak{D}.
\]

Let $I_3, H_3, B_3, C_3, D_3, E_3$ denote the $3\times 3$ permutation matrices corresponding to the permutations $123$, $321$, $132$, $213$, $231$, and $312$, respectively. These six matrices fall into two dihedral symmetry classes: $I_3$ and $H_3$ form one class, while $B_3, C_3, D_3, E_3$ form the other. Combining Lemma~\ref{lm123}, Lemma~\ref{lm132}, and Lemma~\ref{lemma:symmetry} yields the following.

\begin{corollary}\label{cor:3x3-sym}
\[
    \M(n,I_3) = \M(n,H_3)
    \quad\text{and}\quad
    \M(n,B_3) = \M(n,C_3) = \M(n,D_3) = \M(n,E_3).
\]
\end{corollary}

\begin{theorem} \label{3x3}
Let $P$ be a $3\times 3$ permutation matrix. Then
\[
    \M(n,P) = n^2 - 3n + 3.
\]
\end{theorem}

\begin{proof}
The upper bound 
follows from Lemma~\ref{lm123}, Lemma~\ref{lm132}, and Corollary~\ref{cor:3x3-sym}.
For the lower bound, observe that the following $n\times n$ matrices with $n^2-3n+3$ $1$-entries
\[
    S_n = 1 \oplus (J_{n-1} - H_{n-1}) \quad \text{and} \quad
    T_n = 1 \oplus (J_{n-1} - I_{n-1}),
\]
are strongly $I_3$-forcing and strongly $C_3$-forcing, respectively. Thus the statement follows.
\end{proof}

\begin{example}
A $5\times 5$ strongly $I_3$-forcing $(0,1)$-matrix with $5^2-3\cdot 5+3=13$ $1$-entries
\[S_5=\left[\begin{array}{c|c|c|c|c}
  1&0&0&0&0\\ \hline
  0&1&1&1&0\\ \hline
  0&1&1&0&1\\ \hline
  0&1&0&1&1\\ \hline
  0&0&1&1&1
  \end{array}\right]\]

and a $5\times 5$ strongly $C_3$-forcing $(0,1)$-matrix with $13$ $1$-entries
\[S_5=\left[\begin{array}{c|c|c|c|c}
  1&0&0&0&0\\ \hline
  0&0&1&1&1\\ \hline
  0&1&0&1&1\\ \hline
  0&1&1&0&1\\ \hline
  0&1&1&1&0
  \end{array}\right]\]

\end{example}

It is straightforward to verify that
\[
    S_{n,k} = I_{k-2}\oplus \bigl(J_{\,n-k+2} - H_{\,n-k+2}\bigr).
\]
is strongly $I_k$-forcing and contains 
\[
    (n-k+2)^2 - (n-k+2) + (k-2)=n^2-(2k-3)n-(2k-k^2)
\]
$1$-entries, leading to inequality
$$
\M(n,I_k)\ge n^2-(2k-3)n-(2k-k^2).
$$
Moreover, Lemma~\ref{lemma:MI2} and Theorem~\ref{3x3} show that equality holds for $k=2$ and $k=3$, respectively.
This motivates the following conjecture.

\begin{conjecture}\label{conj1}
Let $I_k$ be the $k\times k$ identity matrix. If $n\ge k\geq 3$, then
\[
    \M(n,I_k)=n^2-(2k-3)n-(2k-k^2).
\]
\end{conjecture}

The best we have is the following simple bound.
\begin{lemma}
Let $n\ge k$. For every $k\times k$ permutation matrix $P$, we have $\M(n,P)\le n^2-(k-1)n$.

\end{lemma}
\begin{proof}
    Since every row of $P$ has $k-1$ $0$-entries, every row of an $n\times n$ strongly $P$-forcing matrix has at least as many $0$-entries. Hence, the inequality follows.
\end{proof}

\section{Conclusion}

In this paper, we introduced two notions of pattern forcing in $(0,1)$-matrices: $Q$-forcing and strongly $Q$-forcing, and we studied the corresponding extremal functions $\mathrm{m}(m,n,Q)$ and $\mathrm{M}(m,n,Q)$, respectively. For $Q$-forcing matrices, we showed that the unique extremal construction can be determined explicitly by a simple algorithm, and we characterized its structure using corner functions expressed via Young diagrams. These tools yield closed formulas and monotonicity properties, illustrating that the behavior of $\mathrm{m}(m,n,Q)$ is relatively well-behaved.

In contrast, the structure of strongly $Q$-forcing matrices is substantially more complicated. The extremal function $\mathrm{M}(m,n,Q)$ does not exhibit the monotonicity enjoyed by $\mathrm{m}(m,n,Q)$, and local pattern constraints propagate in less predictable ways. Nevertheless, we proved that in general $mn - \mathrm{M}(m,n,Q) = O(m+n)$, and we determined the exact values of $\mathrm{M}(n,Q)$ for every $2\times 2$ and $3\times 3$ permutation matrix. These results reveal rich symmetry phenomena and highlight a sharp separation between forcing and strong forcing.

Motivated by the exact formulas obtained for small patterns, we proposed a conjecture for $\mathrm{M}(n,I_k)$, suggesting a broader structural principle governing the extremal behavior of strongly $Q$-forcing matrices. We hope that further progress on this conjecture will deepen the understanding of how prescribed subpatterns influence the global density of $(0,1)$-matrices.






\bibliographystyle{elsarticle-num} 
\bibliography{pattern_bib}






\end{document}